\documentclass[fleqn,preprint,3p,a4paper]{elsarticle}
\usepackage{amsmath,amssymb,amsthm,mathtools}
\usepackage{lmodern}
\usepackage{enumitem}
\usepackage{dcolumn}
\usepackage{endnotes}
\usepackage{tabularx}
\usepackage[matrix,arrow]{xy}
\usepackage{wasysym}
\usepackage{graphicx}
\usepackage{microtype}
\usepackage[hidelinks]{hyperref}
\usepackage[T1]{fontenc}

\newtheorem{theorem}{Theorem}[section]
\newtheorem{proposition}[theorem]{Proposition}
\newtheorem{lemma}[theorem]{Lemma}
\newtheorem{corollary}[theorem]{Corollary}
\theoremstyle{remark}
\newtheorem{remark}[theorem]{Remark}

\newcommand{\OO}{\mathcal{O}}
\newcommand{\GG}{\Gamma}
\newcommand{\PH}{P_{H}}

\newcommand{\ua}{{\uparrow}}

\begin{document}

\begin{frontmatter}

\title{A countable metric space whose Hoare power space is not co-sober\tnoteref{t1}}
\tnotetext[t1]{This research was supported by the National Natural Science Foundation of China (Nos. 12471070, 12071199, 12561083).}

\author[mnu]{Xiaoquan Xu\corref{cor}}
\ead{xiqxu2002@163.com}

\cortext[cor]{Corresponding author}

\author[glut]{Wei Ji}
\ead{javeey@163.com}

\address[mnu]{School of Mathematics and Statistics, Minnan Normal University, Zhangzhou 363000, China}
\address[glut]{School of Mathematics and Statistics, Guilin University of Technology, Guilin 541004, China}

\begin{abstract}
Xu asked whether the Hoare power space of a co-sober space, and in particular of a $T_2$-space or a metric space, must be co-sober. We give a negative answer. He and Zhao recently constructed a countable $T_{1}$ space $X$ whose lattice of open sets has a non-sober Scott topology. We observe that their topology has a countable subbase. The classical Ponomarev construction then yields an open continuous surjection from a countable zero-dimensional metrizable space $M$ onto $X$. Direct and inverse images then make the Scott space $\Sigma\mathcal{O}(X)$ a retract of $\Sigma\mathcal{O}(M)$, so $\Sigma\mathcal{O}(M)$ is not sober. Finally, we show that, for every $T_{0}$-space $Y$, the Hoare power space $\PH(Y)$ is co-sober if and only if the Scott space $\Sigma\mathcal{O}(Y)$ is sober. Consequently, $\PH(M)$ is not co-sober. Therefore, the same example answers negatively a question of He and Zhao about Scott sobriety of the lattices of open sets of countable Hausdorff spaces.
\end{abstract}


\end{frontmatter}

\section{Introduction}

Co-sobriety was introduced by Escard\'o, Lawson and Simpson in connection with a dual form of the Hofmann--Mislove theorem \cite{ELS2004}. Sobriety and co-sobriety are not formal duals. Sober spaces need not be co-sober (see \cite[Example 2.1]{WenXu2018}). Escard\'o, Lawson and Simpson asked whether every sober space is co-sober \cite[Problem 9.7]{ELS2004}. In \cite{WenXu2018}, Wen and Xu gave a negative answer to this problem by proving that Isbell's lattice constructed in \cite{Isbell-1982} equipped with the lower topology is sober but not co-sober. Based on Wen and Xu's result, Shen et al. \cite{ShenWuXiZhao2020} further obtained a dcpo whose Scott topology is sober but not co-sober.

The Hoare power space is a very important structure in domain theory, which plays a fundamental role in modeling the semantics of non-deterministic programming languages (see \cite{Abramsky-Jung-1994, GHKLMS-2003, Schalk1993}). There naturally arises a question of which topological properties are preserved by the Hoare power construction.

For a $T_{0}$-space $X$, let $\GG_{0}(X)$ be the set of all nonempty closed subsets of $X$. The Hoare power space $\PH(X)$ is $\GG_{0}(X)$ equipped with the lower Vietoris topology, whose subbasic open sets are
\[
  \Diamond U=\{F\in\GG_{0}(X):F\cap U\neq\varnothing\},
  ~ U\in\OO(X).
\]
This construction is standard in domain theory and non-deterministic semantics; see, for example, \cite{Schalk1993}.

\renewcommand{\thefootnote}{\fnsymbol{footnote}}

In an invited academic presentation\footnote[2]{X. Xu, Some open problems on co-sober spaces,
an invited academic presentation at the 11th International Symposium on Domain Theory and its Applications, Chengdu, China, 4--8 August 2026.}, Xu asked whether the Hoare power space of a co-sober space, more specifically, whether the Hoare power space of a $T_2$-space or a metric space is co-sober. 

We will construct a countable zero-dimensional metrizable space $M$ for which $\PH(M)$ is not co-sober. Since every Hausdorff space is co-sober, this gives a negative answer to these three questions posed by Xu. The same construction also yields
\[
  \Sigma\OO(M)~\text{not sober},
\]
and therefore answers a recent question of He and Zhao concerning countable Hausdorff spaces \cite[Problem 7.3]{HeZhao2026}.

The proof has three parts. First, complementation identifies $\PH(X)$ with the lower topology on the proper open subsets of $X$. A theorem of Wen and Xu then relates co-sobriety of that lower topology to sobriety of $\Sigma\OO(X)$ \cite{WenXu2018}. Second, an open continuous surjection $f:M\to X$ makes $\Sigma\OO(X)$ a retract of $\Sigma\OO(M)$. Third, the Ponomarev construction can be made countable when $X$ is countable and second-countable. We apply these reductions to the space of He and Zhao. All three reductions are proved below; the two non-elementary inputs are the theorem of Wen and Xu for complete lattices and the countable $T_{1}$ example of He and Zhao.

\section{Preliminaries}

For a poset $P$ and $A \subseteq P$, define ${\ua} A=\{x\in P: a\leq x \mbox{ for some }a\in A \}$. For $x\in P$, let ${\ua} x={\ua}\{x\}$.
We use $\mathcal O(X)$ (resp., $\Gamma(X)$) to denote the set of all open subsets (resp., all closed subsets) of a topological space $X$. A nonempty subset $A$ of $X$, $A$ is \emph{irreducible} if for any two closed subsets  $F_1$ and $F_2$, $A \subseteq F_1\cup F_2$ implies $A \subseteq F_1$ or $A \subseteq  F_2$. The space $X$ is called \emph{sober}, if for every irreducible closed subset $F$ of $X$, there is a unique point $a\in X$ such that $F=\overline{\{a\}}$.

The following result is well-known (see, e.g., \cite[Exercise O-5.16(vi)]{GHKLMS-2003}).

\begin{lemma}\label{lem:retract-sober}
Every retract of a sober space is sober.
\end{lemma}

For a $T_0$-space $Y$, let $\leq_Y$ denote the \emph{specialization order} of $Y$: $x\leq_Y y$ if{}f $x\in \overline{\{y\}}$. A subset $A$ of $X$ is called \emph{saturated} if $A$ equals the intersection of all open sets containing it (equivalently, $A$ is an upper set in the specialization order). A nonempty compact saturated subset of $X$ is called $k$-\emph{irreducible} if it cannot be written as the union of two proper compact saturated subsets. The space is \emph{co}-\emph{sober} when every $k$-irreducible compact saturated subset is the saturation of a unique point.

\begin{lemma}\label{lem:T2-co-sober} Every $T_2$-space is co-sober.
\end{lemma}

\begin{proof}
Let $K$ be a $k$-irreducible compact subset of a Hausdorff space $X$. If $x,y\in K$ are distinct, choose disjoint open sets $U$ and $V$ with $x\in U$ and $y\in V$. Then
\[
  K=(K\setminus U)\cup(K\setminus V).
\]
Both sets on the right are proper compact subsets of $K$, and all subsets are saturated because the space is $T_{1}$. This contradicts $k$-irreducibility. Thus $K$ is a singleton, which is the saturation of its unique point.
\end{proof}

\begin{lemma}[Wen--Xu \cite{WenXu2018}]\label{lem:ssturate-closed-co-sober}
Every saturated subspace of a co-sober space is co-sober.
\end{lemma}

A nonempty subset $D$ of a poset $Q$ is \emph{directed} if every two
elements in $D$ have an upper bound in $D$. The set of all directed sets of $Q$ is denoted by $\mathcal D(Q)$. A subset $U$ of $Q$ is \emph{Scott open} if (i) $U=\mathord{\uparrow}U$, and (ii) for any directed subset $D$ for
which $\vee D$ exists, $\vee D\in U$ implies $D\cap
U\neq\emptyset$. All Scott open subsets of $Q$ form a topology,
and we call this topology  the \emph{Scott topology} on $Q$ and
denote it by $\sigma(Q)$. The space $\Sigma~\!\! Q=(Q,\sigma(Q))$ is called the
\emph{Scott space} of $Q$. The \emph{lower topology} on $Q$, generated
by $\{Q\setminus {\uparrow}x : x\in Q\}$ (as a subsase), is denoted by $\omega (Q)$. 

\section{The Hoare power space and the lattice of open sets}

We use the following theorem.

\begin{theorem}[Wen--Xu \cite{WenXu2018}]\label{thm:complete-lattice}
For every complete lattice $L$, the Scott space $\Sigma L$ is sober if and only if $(L,\omega(L))$ is co-sober.
\end{theorem}

We first note that deleting the greatest element does not affect co-sobriety of the lower topology.

\begin{lemma}\label{lem:delete-top}
Let $L$ be a complete lattice with greatest element $1$, and put $L^{-}=L\setminus\{1\}$. Then
\[
  (L,\omega(L))\text{ is co-sober}
  ~\Longleftrightarrow~
  (L^{-},\omega(L^{-}))\text{ is co-sober}.
\]
\end{lemma}

\begin{proof}
Assume that $(L,\omega(L))$ is co-sober. The equality $\omega(L^{-})=\omega(L)|_{L^{-}}$ follows directly from the two subbases. Since the specialization order of $\omega(L)$ is the reverse lattice order, $L^{-}$ is a saturated subspace of $(L,\omega(L))$, and hence by Lemma \ref{lem:ssturate-closed-co-sober}, $(L^{-},\omega(L^{-}))$ is co-sober.

Conversely, suppose that $(L^{-},\omega(L^{-}))$ is co-sober, and let $K$ be a $k$-irreducible compact saturated subset of $(L,\omega(L))$. If $1\in K$, then $K$ is a lower set in the lattice order, and therefore
\[
  K=L=\mathord{\downarrow}_{L}1,
\]
which is the saturation of $1$ in the lower topology. If $1\notin K$, then $K\subseteq L^{-}$. Compactness and saturation pass to the subspace $L^{-}$. Moreover, a compact saturated decomposition of $K$ in $L^{-}$ is also one in $L$, because $L^{-}$ is saturated. Thus $K$ is $k$-irreducible in $L^{-}$ and
\[
  K=\mathord{\downarrow}_{L^{-}}x
\]
for some $x\in L^{-}$. The representing point is unique since the lower topology of a poset is $T_0$. As $x\neq 1$, this set equals $\mathord{\downarrow}_{L}x$. Hence $K$ is principal in $(L,\omega(L))$.
\end{proof}

\begin{theorem}\label{theor:bridge}
For every $T_{0}$-space $X$,
\[
  \PH(X)\text{ is co-sober}
  ~\Longleftrightarrow~
  \Sigma\OO(X)\text{ is sober}.
\]
\end{theorem}

\begin{proof}
Let $\OO^{-}(X)=\OO(X)\setminus\{X\}$. Complementation gives a bijection
\[
  \varphi:\GG_{0}(X)\longrightarrow\OO^{-}(X),
  ~ \varphi(A)=X\setminus A.
\]
For every proper open set $U$,
$$\begin{array}{lll}
	\varphi(\Diamond U)& =\varphi(\{A\in\GG_0 (X) : A\cap U\neq\emptyset\})\\
	           & =\{X\setminus A : A\in\GG_0 (X), U\not\subseteq X\setminus A\}\\
	           & =\{V\in\OO^{-}(X):U\nsubseteq V\}\\
	           & =\OO^{-}(X)\setminus\mathord{\uparrow}U.\\
\end{array}$$
The subbasic set $\Diamond X$ is the whole space. It follows that $\varphi$ is a homeomorphism from $\PH(X)$ onto $(\OO^{-}(X),\omega(\OO^{-}(X)))$.

The lattice $\OO(X)$ is complete. Lemma~\ref{lem:delete-top}, followed by Theorem~\ref{thm:complete-lattice}, gives
\begin{align*}
  \PH(X)\text{ co-sober}
  &\Longleftrightarrow
  (\OO^{-}(X),\omega(\OO^{-}(X)))\text{ co-sober}\\
  &\Longleftrightarrow
  (\OO(X),\omega(\OO(X)))\text{ co-sober}\\
  &\Longleftrightarrow
  \Sigma\OO(X)\text{ sober}.
\end{align*}
\end{proof}

\begin{remark}
Theorem~\ref{theor:bridge} is a transfer statement. Its substantive input is Theorem~\ref{thm:complete-lattice}; complementation and Lemma~\ref{lem:delete-top} put that theorem into the form needed for Hoare power spaces.
\end{remark}

\section{Open maps and a countable Ponomarev construction}

We next transfer non-sobriety of lattices of open sets along open surjections.

\begin{lemma}\label{lem:open-retract}
Let $f:M\to X$ be an open continuous surjection. Then $\Sigma\OO(X)$ is a retract of $\Sigma\OO(M)$.
\end{lemma}

\begin{proof}
Define
\[
  s:\OO(X)\to\OO(M),~ s(U)=f^{-1}(U),
\]
and
\[
  r:\OO(M)\to\OO(X),~ r(V)=f(V).
\]
The second map is well defined because $f$ is open. Both maps preserve arbitrary unions, hence directed suprema, and are Scott continuous. Surjectivity of $f$ gives
\[
  r(s(U))=f(f^{-1}(U))=U
\]
for every $U\in\OO(X)$. Therefore $r\circ s=id_{\OO(X)}$.
\end{proof}

By Lemma \ref{lem:retract-sober} and Lemma \ref{lem:open-retract}, we get the following.

\begin{corollary}\label{cor:transfer-nonsober}
Under the hypotheses of Lemma~\ref{lem:open-retract}, if $\Sigma\OO(X)$ is not sober, then $\Sigma\OO(M)$ is not sober.
\end{corollary}


Hanai and Ponomarev characterized first-countable $T_{1}$-spaces as open continuous images of metrizable spaces; see \cite[Theorem~2.1]{Abernethy1974}. For the present application, the standard construction can be chosen countable.

\begin{theorem}\label{theor:countable-ponomarev}
Every countable second-countable $T_{1}$-space is an open continuous image of a countable zero-dimensional metrizable space.
\end{theorem}

\begin{proof}
Let $X$ be countable and second-countable, and let $\mathcal{B}$ be a countable base. Regard $\mathcal{B}$ as a discrete space. For a finite sequence
\[
  t=(B_{0},\ldots,B_{k-1})\in\mathcal{B}^{<\mathbb{N}},
\]
put $\bigcap t=X$ when $k=0$ and $\bigcap t=\bigcap_{i<k}B_{i}$ otherwise.

For every pair $(x,t)$ with $x\in\bigcap t$, choose a sequence
\[
  b(x,t)=(B_{n}(x,t))_{n\in\mathbb{N}}\in\mathcal{B}^{\mathbb{N}}
\]
whose first $k$ terms are those of $t$ and whose terms form a neighborhood base at $x$. Such a sequence exists because a second-countable space is first-countable. Set
\[
  M=\{b(x,t):t\in\mathcal{B}^{<\mathbb{N}},\ x\in\bigcap t\}
  \subseteq\mathcal{B}^{\mathbb{N}}.
\]
There are only countably many pairs $(x,t)$, so $M$ is countable. The cylinder sets form a clopen base of $\mathcal{B}^{\mathbb{N}}$; hence $M$ is zero-dimensional and metrizable. For example, the product is metrized by
\[
  d(a,b)=
  \begin{cases}
    0, & a=b,\\
    2^{-\min\{n~\!\!:~\!\!a_{n}\neq b_{n}\}}, & a\neq b.
  \end{cases}
\]

A sequence in $M$ cannot be a neighborhood base at two distinct points. If it were a base at $x\neq y$, then every term would contain both points. The $T_{1}$ open neighborhood $X\setminus\{y\}$ of $x$ would contain one of those terms, a contradiction. We may therefore define
\[
  p:M\to X,~ p(b(x,t))=x.
\]
Taking $t$ to be the empty sequence shows that $p$ is surjective.

For $t=(B_{0},\ldots,B_{k-1})$, let
\[
  [t]=\{a\in\mathcal{B}^{\mathbb{N}}:a_{i}=B_{i}\text{ for }i<k\}.
\]
Then
\[
  p(M\cap[t])=\bigcap t. \tag{1}
\]
The inclusion from left to right holds because every term of a sequence assigned to $x$ is a neighborhood of $x$. The reverse inclusion follows from the presence of $b(x,t)$ for every $x\in\bigcap t$. The sets $M\cap[t]$ form a base of $M$, and (1) shows that $p$ is open.

Finally, let $a=(A_{n})_{n\in\mathbb{N}}\in M$, let $p(a)=x$, and let $U$ be an open neighborhood of $x$. Since $(A_{n})$ is a neighborhood base at $x$, there is an $n$ with $A_{n}\subseteq U$. For $t=(A_{0},\ldots,A_{n})$, equation (1) gives
\[
  p(M\cap[t])=\bigcap_{i\leq n}A_{i}\subseteq A_{n}\subseteq U.
\]
Thus $p$ is continuous.
\end{proof}

\section{The countable metric counterexample}

We now extract the precise input needed from the construction of He and Zhao.

\begin{proposition}\label{prop:HZ-input}
There exists a countable second-countable $T_{1}$-space $X$ such that $\Sigma\OO(X)$ is not sober.
\end{proposition}

\begin{proof}
Corollary 4.13 of \cite{HeZhao2026} gives a countable $T_{1}$-space that is compact but not $SI$-compact. Lemma 3.5 of the same paper states that, if $\Sigma\OO(X)$ is sober, then compactness and $SI$-compactness are equivalent. Hence the Scott space $\Sigma\OO(X)$ of their example is not sober.

It remains only to record a countability property of their topology. The underlying poset $\widehat P$ is countable. Its lower topology is generated by the countable subbase
\[
  \{\widehat P\setminus\mathord{\uparrow}x:x\in\widehat P\}.
\]
The topology in Corollary 4.13 is generated by this lower topology together with the co-finite subsets of $\widehat P$. There are only countably many co-finite subsets of a countable set. The resulting topology therefore has a countable subbase, and is second-countable.
\end{proof}

\begin{theorem}\label{thm:main}
There exists a countable zero-dimensional metrizable space $M$ such that
\[
  \Sigma\OO(M)\text{ is not sober}
  ~\text{and}~
  \PH(M)\text{ is not co-sober}.
\]
\end{theorem}

\begin{proof}
Choose $X$ as in Proposition~\ref{prop:HZ-input}. By Theorem~\ref{theor:countable-ponomarev}, there are a countable zero-dimensional metrizable space $M$ and an open continuous surjection $p:M\to X$. Corollary~\ref{cor:transfer-nonsober} shows that $\Sigma\OO(M)$ is not sober. Theorem~\ref{theor:bridge} then gives that $\PH(M)$ is not co-sober.
\end{proof}

\section{Concluding remarks}

The proof combines two transfer steps. An open continuous surjection transfers non-sobriety of the Scott space of the lattice of open sets from the He--Zhao space to a countable metrizable domain. Theorem~\ref{theor:bridge} then converts that non-sobriety into the failure of co-sobriety for the Hoare power space. The resulting example is countable and zero-dimensional.

He and Zhao also asked whether $\Sigma\OO(\mathbb{Q})$ is sober for the rational line $\mathbb{Q}$. The construction above does not decide that more specific question. By Theorem~\ref{theor:bridge}, it is equivalent to asking whether $\PH(\mathbb{Q})$ is co-sober.






\begin{thebibliography}{99}

\bibitem{Abernethy1974}
K. Abernethy, On characterizing certain classes of first countable spaces by open mappings,
Pac. J. Math. 53 (1974), 319--326.

\bibitem{Abramsky-Jung-1994}
S. Abramsky, A. Jung, Domain theory, in: Semantic Structures, Handbook of Logic in Computer Science, vol.3, Clarendon Press, 1994, pp.1-168.




\bibitem{ELS2004}
M. Escard\'o, J. D. Lawson, A. Simpson,
Comparing Cartesian closed categories of (core) compactly generated spaces,
\emph{Topology Appl.} 143 (2004), 105--145.


\bibitem{Erne2019}
M. Ern\'e, Web spaces and worldwide web spaces: topological aspects of domain theory,
Log. Meth. Comput. Sci. 15(1) (2019), 23:1--23:38.

\bibitem{GHKLMS-2003}G. Gierz, K. Hofmann, K. Keimel, J.D. Lawson, M. Mislove, D. Scott, Continuous Lattices and Domains, Encyclopedia of Mathematics and its Applications 93. Cambridge University Press, 2003.


\bibitem{HeZhao2026}
Z. He, B. Zhao, A new method for constructing a compact and non SI-compact space, 2026,
arXiv:2608.01048.

\bibitem{Isbell-1982}
J. Isbell, Completion of a construction of Johnstone, Proceedings of the American Mathematical Society, 85 (1982) 333-334.


\bibitem{Schalk1993}
A. Schalk, Algebras for Generalized Power Constructions, PhD thesis, Technische Hochschule Darmstadt, 1993.

\bibitem{ShenWuXiZhao2020}
C. Shen, G. Wu, X. Xi and D. Zhao, Sober Scott spaces are not always co-sober,
Topol. Appl. 282 (2020) 107316.

\bibitem{WenXu2018}
X. Wen, X. Xu, Sober is not always co-sober,
Topol. Appl. 250 (2018), 48--52.



\bibitem{ZhaoHeWang2023}
B. Zhao, Z. He, K. Wang, On co-sober spaces,
Houst. J. Math. 49(4) (2023), 971--988.

\end{thebibliography}
\end{document}